\documentclass[11pt]{article}

\usepackage{amsmath, amsthm, amssymb}
\usepackage{graphicx}    
\usepackage{palatino}

\usepackage{url, hyperref}

\newtheorem{thm}{Theorem}[section]

\newtheorem{cor}[thm]{Corollary}

\newtheorem{lem}[thm]{Lemma}

\title{\textbf{Finding hitting times in various graphs.}
\author{ Shravas K Rao\footnote{Work supported by Massachusetts Institute of Technology's Undergraduate Research Opportunities Program} \\ Massachusetts Institute of Technology \\ \texttt{shravas@mit.edu}}}

\begin{document}

\maketitle

\begin{abstract}
The hitting time, $h_{uv}$, of a random walk on a finite graph $G$, is the expected time for the walk to reach vertex $v$ given that it started at vertex $u$.  We present two methods of calculating the hitting time between vertices of finite graphs, along with applications to specific classes of graphs, including grids, trees, and the 'tadpole' graphs.
\end{abstract}

\textbf{keywords:} random walks, hitting time

\section{Introduction}

A random walk on a graph is a walk that begins at a particular starting vertex in which each successive step is determined by randomly choosing an edge adjacent to the previous vertex and traveling to the vertex at the other endpoint of the edge.  This random choice is distributed equally over all edges adjacent to the vertex.  For our purposes, we will consider only random walks on unweighted and undirected graphs.  Although much of the historical work on random walks considers infinite graphs, recent work has dealt more with finite graphs.  When dealing with random walks on finite graphs, the focus turns to less qualitative questions; rather than asking whether or not a random walk will return to its starting vertex, it may be interesting to ask how long the random walk would take to return to the starting vertex.

A property that arises from analyzing random walks on finite graphs is the hitting time.  Given a finite graph $G$, the hitting time, $h_{uv}$,  from a vertex $u$ to a vertex $v$, is the expected number of steps it takes for a random walk that starts at vertex $u$ to reach vertex $v$.  Note that when dealing with finite graphs, the hitting time from $u$ to $v$ is finite if and only if the vertices $u$ and $v$ are connected.

The hitting time of a random walk has many useful properties.  For example, the cover time, the expected time it takes for a random walk to visit all vertices of a graph, can be both bounded above by a function of the largest hitting time from one vertex to another, and below by a function of the smallest hitting time from one vertex to another~\cite{randomwalk}.  However, the hitting times between vertices of various graphs can be hard to analyze, and finding their values is not intuitive.  There exist a few known bounds on hitting times, as found in~\cite{c, b, Aldous, a}.  We will focus on finding exact formulas for the hitting time for certain types of graphs.

We will consider two different methods of calculating the hitting time of a graph and demonstrate their applications to various classes of graphs.  Section 2 gives an explicit formula for finding the hitting time, but only in a few specific cases.  Specifically, the formula can only be used to find the hitting time from one vertex to a neighbor and only if the graph exhibits a symmetry about the starting vertex.  This can then be applied to random walks on a variety of classes of graphs, including grids, hypercubes, and trees.

Section 3 uses the method of calculating hitting times through a system of linear equations first shown in~\cite{Aldous}.  Although this is not as convenient as a formula, this method can be applied to random walks on any graph.  We can use this to derive formulas for hitting times of random walks on graphs for which the method of Section 2 does not apply, such as for the complete $d$-ary tree, or the tadpole graph.

\section{Hitting times in graphs with symmetry}

\subsection{Proof of theorem}


The following theorem gives a formula that can, in certain cases, be used to find hitting times.

\begin{thm}\label{recurrence}
Let $v$ be a vertex of a connected graph $G$ with neighbor $u$.  If for every other neighbor of $v$, there exists an automorphism of $G$ that maps $u$ to that neighbor of $v$, then the hitting time from $u$ to $v$ is $\frac{2e}{k}-1$, where $e$ is the number of the edges in the graph, and $k$ is the number of neighbors of $v$.
\end{thm}

\begin{proof}
Because of the symmetry of the graph, the hitting time is equal from any neighbor of $v$ to $v$.  Let this value be $x$.

Now consider the random walk on $G$ that starts at vertex $v$ and moves to vertex $u$ in the next step, and the next such $t$ in which the $t$th vertex visited is $v$ and the $(t+1)$th vertex visited is $u$.  We can find the expected value of the next such $t$, both in terms of both $k$ and $x$, and in terms of $e$, allowing us to solve for $x$.  This can also be thought of as finding the recurrence time of a random walk along the edges of the directed graph $G'$, formed by replacing each edge in $G$ with two, one in each direction.

Consider the structure of the walk more closely.  After starting at $v$ and then going to $u$, in order to again return back to $v$ and then $u$, the walk must first return to $v$.  This takes an expected $x$ steps.  At this point, the walk can continue in two different ways; the walk may go to $u$, with a probability of $\frac{1}{k}$ and $1$ additional step.  It may also continue on to another vertex, in which case the walk must again return back to $v$. This adds on average, another $x+1$ to the number of steps the random walk has taken.  Continuing in this manner, we find that that the expected recurrence time is

\[\frac{1}{k}\left(x+1\right)+\frac{1}{k}\frac{k-1}{k}2\left(x+1\right)+\frac{1}{k}\frac{k-1}{k}\frac{k-1}{k}3\left(x+1\right)+\cdots\]

which simplifies to $k\left(x+1\right)$.  

It is well known that the expected recurrence time in this walk is $2e$, as there are $2e$ edges in $G'$~\cite{randomwalk}.  Setting these two values equal and solving for $x$ gives $x = \frac{2e}{k}-1$.
\end{proof}

\subsection{Applications to the grid, hypercube, and trees}

This technique can be used to find the hitting time in the following graphs. 

\begin{cor}
In a $d$-dimensional grid whose dimensions all have length $m$, the hitting time from a corner to one of its neighbors is $2\left(m-1\right)m^{d-1}-1$
\end{cor}

\begin{proof}
There are $d\left(m-1\right)m^{d-1}$ edges in a $d$-dimensional grid whose dimensions all have length $m$.  Applying Theorem \ref{recurrence} gives the hitting time given above.
\end{proof}

In particular, the application of this technique to the hypercube gives a more general result, as every vertex is a corner.

\begin{cor}
In a $d$-dimensional hypercube, the hitting time from any vertex to one of its neighbors is $2^d-1$.
\end{cor}

Additionally, this technique can be used to find the hitting times from one vertex to a neighbor in trees.  

\begin{cor}\label{tree}
The hitting time from a vertex $v$ to a neighbor $u$ in a tree is $2e-1$, where $e$ is the number of edges in the connected component containing $u$ after the edges to all other neighbors of $u$ are removed.
\end{cor}

\begin{proof}
Removing the edges to all other neighbors of $u$ does not change the hitting time from $v$ to $u$, as it is not possible for a random walk starting at $v$ to reach those vertices without having already reached $u$.  However, this allows the hypothesis of Theorem \ref{recurrence} to hold for the connected component containing $u$, as $u$ now only has $1$ neighbor.  Finally, we can just apply the result to this modified graph.
\end{proof}

\section{Hitting time via a system of linear equations}
The second technique we will use to find hitting times can be applied to any graph, rather than graphs with some property of symmetry, but does not yield formulas as easily.  The following theorem, first shown in~\cite{Aldous} gives a set of linear equations whose solution gives the hitting times of a random walk on a graph.  If we know the general structure of the graph, then it may be possible to find the hitting times in terms of certain properties of the graph.

\begin{thm}\label{eq}
In a graph $G$, let $h_{ij}$ be the hitting time from a vertex $i$ to a vertex $j$ and let $\Gamma\left(i\right)$ be the set of neighbors of $i$.  The following set of equations, for a fixed vertex $j$ and all vertices $i$ in $G$

\[h_{ij} = \left\{
     \begin{array}{lr}
       0 &  i=j\\
       1+\frac{1}{|\Gamma\left(i\right)|}\displaystyle\sum_{k\in\Gamma\left(i\right)}{h_{kj}} &  i \neq j
     \end{array}
   \right.\]

yields a unique solution for all $h_{ij}$.
\end{thm}


Note that we can use this theorem to both construct hitting times by solving the system of equations, or to show that a given formula for the all the hitting times to a certain vertex satisfy the set of equations, as they describe exactly one solution.

\subsection{Application to trees}
We first use this technique to give an alternate proof to Corollary \ref{tree}.

\begin{thm}
Let $v$ and $u$ be neighbors of a tree rooted at $u$.  Let $n$ be the number of vertices in the subtree rooted at $v$.  Then the hitting time from $v$ to $u$ is $2n-1$.
\end{thm}

\begin{proof}
This can be proven using induction.  The base case where $n$ is $1$ can be checked by calculation.

By Theorem \ref{eq}, the following must hold

\[h_{vu} = 1+\frac{1}{|\Gamma\left(i\right)|}\displaystyle\sum_{j\in\Gamma\left(i\right)}{h_{ju}}.\]

One of the neighbors of $v$ is $u$, in which case the hitting time to $u$ is $0$.  For every other neighbor of $v$, $j$, the hitting time from $j$ to $v$ can also be expressed as $h_{jv}+h_{vu}$.  Letting $n_j$ be the number of nodes in the subtree rooted at $j$, by the inductive hypothesis, $h_{jv} = 2n_j-1$.  We are left with the following equation

\[h_{vu} = 1+\frac{1}{|\Gamma\left(i\right)|}\left(\displaystyle\sum_{j\in\Gamma\left(i\right)-\{u\}}{2n_j-1}\right)+\frac{|\Gamma\left(i\right)|-1}{|\Gamma\left(i\right)|}h_{vu}.\]

Solving for $h_{vu}$ gives

\[h_{vu} = 1+\displaystyle\sum_{j\in\Gamma\left(i\right)-\{u\}}{2n_j}.\]

Because $1+\displaystyle\sum_{j\in\Gamma\left(i\right)-\{u\}}{n_j}$ is the number of vertices in the subtree rooted at $v$, we can conclude that $h_{vu} = 2n-1$, completing the inductive step.
\end{proof}

\subsection{Application to the tadpole graph}
We can also use these ideas to find the hitting time in a "tadpole" graph, a graph which consists of a cycle attached to a line.  

\begin{thm}
Let $v$ be vertex with the largest distance from the cycle.  The hitting time from $v$'s neighbor $u$ to $v$ is $2l+2k-1$, where $k$ is the number of vertices in the cycle, and $l$ is the number of vertices in the line, not including $a$.
\end{thm}

\begin{proof}

We start by characterizing the hitting times of the vertices on the line.

\begin{lem}

Let the distance from a vertex $u'$ on the line to $v$ be $l$.  Then the hitting from $u'$ to $v$ is $lh_{uv}-l\left(l-1\right)$.

\end{lem}

\begin{proof}

This can be shown using induction.  The base cases, $l=0, 1$ are can be checked by calculation.  Assume the above holds true for all $l$, up to $i$.  Using Theorem \ref{eq}, we know that

\[h_{iv} = \frac{1}{2}\left(h_{\left(i+1\right)v}+h_{\left(i-1\right)v}\right)+1.\]

Replacing the values for $h_{iv}$ and $h_{\left(i-1\right)v}$, we have that

\[h_{uv}-l\left(l-1\right) = \frac{1}{2}\left(h_{\left(i+1\right)v}\left(l-1\right)lh_{uv}-\left(l-1\right)\left(l-2\right)\right)+1.\]

Solving for $h_{\left(i+1\right)v}$ gives the desired result.

\end{proof}

We can then continue by characterizing the hitting times of the vertices on the cycle

\begin{lem}
If $k$ is even, then let $v'$ be the vertex on the cycle farthest away from $v$.  Then, for any vertex on the cycle, $w$, $h_{wv} = h_{v'v}-l^2$, where $l$ is the distance from $w$ to $v'$.

If $k$ is odd, then let $v'$ and $v''$, be the vertices on the cycle farthest away from $v$.  Then, for any vertex on the cycle, $w$, $h_{wv} = h_{v'v}-l\left(l+1\right)$, where $l$  is the smaller of the two distances from $w$ to $v'$ and $v''$.
\end{lem}

\begin{proof}
Both cases can be shown using induction.  When $k$ is even, the base case where $l=0$ can be checked by calculation.  Assume the lemma holds true for all $l$ up to $i$.  Then by Theorem \ref{eq}, we have that

\[h_{iv} = \frac{1}{2}\left(h_{\left(i+1\right)v}+h_{\left(i-1\right)v}\right)+1.\]

Replacing the values for $h_{iv}$ and $h_{\left(i-1\right)v}$, we have that

\[h_{v'v}-i^2 = \frac{1}{2}\left(h_{\left(i+1\right)v} + h_{v'v}-\left(i-1\right)^2\right)+1.\]

Solving for $h_{\left(i+1\right)v}$ gives the desired result.

When $k$ is odd, by a symmetry argument, $h_{v'v} = h_{v''v}$.  Therefore, by Theorem \ref{eq}, if $l = 1$, then we have that

\[h_{v'v} = \frac{1}{2}\left(h_{v''v}+h_{wv}\right)+1,\]

and therefore $h_{wv} = h_{v'v}-2$.

The rest of the proof uses an induction argument identical to that for when $k$ is even.
\end{proof}

Finally, we can take advantage of the two characterizations to solve for $h_{uv}$.  First, consider the vertex on both the cycle, and the line, $c$.  The above lemmas show that the neighbors of $c$ have hitting times of $h_{cv}+k-1, h_{cv}+k-1$, and $h_{cv}-h_{uv}+2\left(l-1\right)$.  Again, by Theorem \ref{eq}, we have that

\[h_{cv} = \frac{1}{3}\left(h_{cv}+k-1\right)+\frac{1}{3}\left(h_{cv}+k-1\right)+\frac{1}{3}\left(h_{cv}-h_{uv}+2\left(l-1\right)\right)+1.\]

Solving for $h_{uv}$ gives $h_{uv} = 2k+2l-1$, as desired.

\end{proof}

\subsection{Application to complete $d$-ary trees}
In other cases, we can use Theorem \ref{eq}, to find hitting times between vertices that are not neighbors.  Consider the complete $d$-ary tree.  

To describe the hitting times between two vertices in a complete $d$-ary tree, we start by describing the hitting time from any vertex to the root.  By symmetry arguments, it follows that the hitting time depends only on the distance away from the root.  To aid in describing these hitting times, we define the following polynomial

\[f_n(d) =  \left\{
     \begin{array}{lr}
       0 & n=0\\
       \left(\displaystyle\sum_{i=0}^{n-1}\left(2n-2i\right)d^i\right)-n & \text{o.w.}
     \end{array}
   \right.\]

for nonnegative integers $n$.

Now, we claim the following

\begin{lem}\label{f}
The hitting time for a vertex of distance $l$ away from the root to the root in a complete $d$-ary tree of height $h$ is

\[f_h\left(d\right)-f_{h-l}\left(d\right).\]
\end{lem}

\begin{proof}

For convenience we will let $f_h$ denote $f_h(d)$.  By Theorem \ref{eq}, it is sufficient to show both

\[f_h-f_{0} = 1+f_h-f_1\]

and

\[f_h-f_{h-l} = 1+\frac{1}{d+1}\left(f_h-f_{h-l+1}\right) + \frac{d}{d+1}\left(f_h-f_{h-l-1}\right)\]

when $l < h$, as these describes all possible equations.  Because $f_1 = 1$, the first equation holds.

Simplifying the second equation, we see that

\[\left(d+1\right)\left(f_{h-l}+1\right) = d\left(f_{h-l-1}\right)+f_{h-l+1}.\]

The following holds for $n$ greater than $1$

\[f_n = df_{n-1}+\left(n-1\right)d+n.\]

We can use this to see that

\[\left(d+1\right)\left(f_{h-l}+1\right) = f_{h-l}-\left(h-l-1\right)d+h-l-1+df_{h-1}+\left(h-l\right)d+h-1.\]

\end{proof}

To continue finding hitting times we calculate the hitting time from each ancestor of a leaf, to that leaf.  To do so, we define the following polynomial

\[g_{k, m}\left(d\right) =  \left\{
     \begin{array}{lr}
       0 & m=0\\
       \left(\displaystyle\sum_{i=0}^{m-1}\left(2m-2i\right)d^{k-i}\right)-m & \text{o.w.}

     \end{array}
   \right.\]

for positive integers $k$ and nonnegative integers $m$.  We claim the following

\begin{lem}\label{g}
Let $v$ be a leaf in a complete $d$-ary tree of height $h$.  The hitting time from the ancestor of distance $l$ from $v$, to $v$, is

\[g_{h, h}\left(d\right)-g_{h, h-l}\left(d\right).\]
\end{lem}

\begin{proof}
If a vertex is not an ancestor of $v$, then any random walk starting at that vertex must pass through an ancestor of $v$ before reaching $v$.  The hitting time can be calculated by adding the relevant hitting times given by Lemma \ref{f} and this theorem.  Therefore, it is sufficient to only consider the equations that describe the hitting time from an ancestor to $v$.

Because the degree of the root of a complete $d$-ary tree is different from the degree of all other non-leaf vertices, we treat this case separately.

The neighbors of the root include one vertex that is an ancestor of $v$, and $d-1$ other vertices.  For the latter, any random walk originating at any of these vertices must pass through to the root in order to reach $v$.  Therefore by Theorem~\ref{eq} and the previous lemma, it is necessary that the following holds

\[g_{h, h}-g_{h, 0} = 1+\frac{1}{d}\left(g_{h, h}-g_{h, 1}\right)+\frac{d-1}{d}\left(g_{h, h}+f_h-f_{h-1}\right).\]

Because $g_{h, 0} =0$, this simplifies to

\[0 = d-g_{h, 1}+\left(d-1\right)\left(f_h-f_{h-1}\right).\]

It can be seen that $\left(d-1\right)\left(f_h-f_{h-1}\right)+d = 2d^h-1$ for all positive $h$, and therefore we have that

\[g_{h, 1} = 2d^h-1,\]

which is true by our definition of $g$.

We continue for all other ancestors of $v$. Let the distance from an ancestor of $v$ to $v$ be $k$.  For these vertices, there are $d+1$ neighbors to consider: $1$ ancestor of $v$ closer to $v$, $1$ ancestor of $v$ farther away, and $d-1$ additional neighbors.  Again, any random walk starting from any vertex in the latter group must pass through this ancestor to reach $v$.  It is necessary that the following holds

\[g_{h, h}-g_{h, k} = 1+\frac{1}{d+1}\left(g_{h, h}-g_{h, k+1}\right)+\frac{1}{d+1}\left(g_{h, h}-g_{h, k-1}\right)+\frac{d-1}{d+1}\left(g_{h, h}-g_{h, k}+f_{h-k}-f_{h-k-1}\right)\]

\[g_{h, k+1}+g_{h, k-1}-2g_{h, k} = d+1+\left(d-1\right)\left(f_{h-k}-f_{h-k-1}\right).\]

Note that $g_{h, k+1}+g_{h, k-1}-2g_{h, k} = 2d^{h-k}$, as each term in $g_{h, k}$ is the average of the corresponding terms in $g_{h, k+1}$ and $g_{h, k-1}$, except for the $d^{h-k}$ term.  The right hand side can be simplified using the same identity used in the previous lemma.  Therefore, we get

\[2d^{h-k} = 2d^{h-k}-1+1,\]

which holds true.

We have shown that all equations in Theorem \ref{eq} referring to the hitting time from an ancestor of $v$ to $v$ hold.  As stated previously, this is enough to prove the lemma.
\end{proof}

This gives all the tools needed to find the hitting time from any vertex to any other vertex in a complete $d$-ary tree.

\begin{thm}
Let $u$ and $v$ be two vertices in a complete $d$-ary tree, with a least common ancestor of $c$.  Let $u'$, $v'$ and $c'$ be the distances of $u$, $v$, and $c$ respectively, to the root, and let $h$ be the height of the tree.  Then the hitting time from $u$ to $v$ is

\[f_{h-c'}\left(d\right)-f_{h-u'}\left(d\right)+g_{h, v'}\left(d\right)-g_{h, c'}\left(d\right).\]
\end{thm}

\begin{proof}
Let $l$ be an arbitrary leaf descended from $v$.  Then any random walk from $u$ to $l$ must also go through $c$ and $v$.  Therefore,

\[h_{ul} = h_{uc}+h_{cl} = h_{uv}+h_{vl}.\]

Then, we can express $h_{uv}$ as follows

\[h_{uv} = h_{uc}+h_{cl}-h_{vl}.\]

Lemma \ref{f} gives the value of $h_{uc}$, and Lemma \ref{g}, gives the values of both $h_{cl}$ and $h_{vl}$.  Replacing these values into the above equation gives the formula for hitting times as stated.
\end{proof}

\bibliographystyle{amsplain}
\bibliography{refHittingTime}		
\end{document}